\documentclass[12pt]{amsart}
\usepackage{amssymb}
\usepackage{amsthm}

\usepackage{amscd}
\usepackage{extarrows}
\usepackage{graphicx}
\usepackage{fancyhdr}
\usepackage{epsfig}
\usepackage{amsfonts}
\usepackage{mathrsfs}
\usepackage{picinpar} \usepackage{amsmath} \usepackage[all]{xy}
\usepackage[colorlinks=true]{hyperref}
\hypersetup{urlcolor=blue, citecolor=red}
\theoremstyle{plain}
\newtheorem{theorem}{Theorem}[section]
\newtheorem{corollary}[theorem]{Corollary}
\newtheorem{lemma}[theorem]{Lemma}
\newtheorem{proposition}[theorem]{Proposition}
\theoremstyle{definition}
\newtheorem{definition}[theorem]{Definition}
\newtheorem{example}[theorem]{Example}
\newtheorem{remark}[theorem]{Remark}
\newcommand{\Ext}{\mbox{\rm Ext}}
\newcommand{\Hom}{\mbox{\rm Hom}}

\newcommand{\Tor}{\mbox{\rm Tor}}

\newcommand{\id}{\mbox{\rm id}}

\newcommand{\Ker}{\mbox{\rm Ker}}

\newcommand{\pd}{\mbox{\rm pd}}
\newcommand{\fd}{\mbox{\rm fd}}

\begin{document}
\title{Gorenstein flat modules with respect to duality pairs}
\author{ZHANPING WANG \ \ \ \ GANG YANG}

\footnote[0]{*Supported by National Natural Science Foundation of China (Grant Nos. 11561061, 11561039)}\footnote[0]{Address
correspondence to Zhanping Wang, Department of Mathematics, Northwest Normal University, Lanzhou 730070, P.R. China.}\footnote[0]{E-mail: wangzp@nwnu.edu.cn (Z.P. Wang),
yanggang@mail.lzjt.cn (G. Yang).}

\date{}\maketitle
\hspace{6.3cm}\noindent{\footnotesize {\bf Abstract}
\vspace{0.2cm}

\hspace{-0.75cm} Let $\mathcal{X}$ be a class of left $R$-modules, $\mathcal{Y}$ be a class of right $R$-modules. In this paper, we introduce and study Gorenstein $(\mathcal{X}, \mathcal{Y})$-flat modules as a common generalization of some known modules such as Gorenstein flat modules \cite{EJT93}, Gorenstein $n$-flat modules \cite{SUU14}, Gorenstein $\mathcal{B}$-flat modules \cite{EIP17}, Gorenstein AC-flat modules \cite{BEI17}, $\Omega$-Gorenstein flat modules \cite{EJ00} and so on. We show that the class of all Gorenstein $(\mathcal{X}, \mathcal{Y})$-flat modules have a strong stability. In particular, when $(\mathcal{X}, \mathcal{Y})$ is a perfect (symmetric) duality pair, we give some functorial descriptions of Gorenstein $(\mathcal{X}, \mathcal{Y})$-flat dimension, and construct a hereditary abelian model structure on $R$-Mod whose cofibrant objects are exactly the Gorenstein $(\mathcal{X}, \mathcal{Y})$-flat modules. These results unify the corresponding results of the aforementioned modules.

\vspace{0.2cm}
\noindent{\footnotesize {2010 {\it{Mathematics Subject
Classification}}:}
18G35, 55U15, 13D05, 16E30

\noindent{\footnotesize {{\it{Keywords and phrases}}:}  Gorenstein $(\mathcal{X}, \mathcal{Y})$-flat modules,  Gorenstein $(\mathcal{X}, \mathcal{Y})$-flat dimension, stability, Gorenstein flat model structure, duality pair

\section{Introduction and Preliminaries}
Enochs and coauthors introduced Gorenstein projective, injective and Gorenstein flat modules, and developed Gorenstein homological algebra in \cite{EJ95,EJT93,EJ00}. Later, many scholars further studied these modules and introduced various generalizations of these modules (See, e.g., \cite{Ben09,BGH14,BK12,Gil17,Hol04,Hu13}). For example, Bravo, Estrada and Iacob defined Gorenstein AC-flat modules in \cite{BEI17}, Estrada, Iacob and P\'{e}res studied Gorenstein $\mathcal{B}$-flat modules in \cite{EIP17}, where $\mathcal{B}$ is a class of right $R$-modules.

In \cite{Gil17}, Gillespie constructed a hereditary abelian model structure, the Gorenstein flat model structure, over a right coherent ring $R$, in which the cofibrant objects are precisely the Gorenstein flat modules. In \cite{EIP17}, Estrada, Iacob and P\'{e}res studied relative Gorenstein flat model structure on the categories of left $R$-modules and complexes.

Let $\mathcal{X}$ be a class of left $R$-modules, $\mathcal{Y}$ be a class of right $R$-modules. In section $2$ of this paper, we define and study Gorenstein $(\mathcal{X}, \mathcal{Y})$-flat modules. A left $R$-module $M$ is called Gorenstein $(\mathcal{X}, \mathcal{Y})$-flat module, if there exists an exact sequence of left $R$-modules in $\mathcal{X}$,
$$\cdots\rightarrow X_{1}\rightarrow X_{0}\rightarrow X^{0}\rightarrow X^{1}\rightarrow\cdots$$
such that $M\cong \Ker(X^{0}\rightarrow X^{1})$ and $Y\otimes_{R}-$ leaves the sequence exact whenever $Y$ in $\mathcal{Y}$. For different choices of $\mathcal{X}, \mathcal{Y}$, the class $\mathcal{GF}_{\mathcal{(X,Y)}}(R)$ of all Gorenstein $(\mathcal{X}, \mathcal{Y})$-flat modules encompasses all of the aforementioned modules, and some results existing in the literature for the modules above can be obtained as particular cases of the results on Gorenstein
$(\mathcal{X}, \mathcal{Y})$-flat modules.

Section $3$ is devoted to a hereditary abelian model structure which is related to Gorenstein $(\mathcal{X}, \mathcal{Y})$-flat modules, where $(\mathcal{X}, \mathcal{Y})$ is a duality pair.

Next we collect all the background material that will be necessary in the sequel. For unexplained concepts and notations, we refer the reader to \cite{EJ00,Gil16,Hov99}.

 \textbf{Cotorsion pairs.} Let $\mathcal{A}$ be a class of left $R$-modules. We define $\mathcal{A}^{\bot}=\{M~|~\Ext_{R}^{1}(A, M)=0 ~\text{for all}~ A\in \mathcal{A}\}$, and $^{\bot}\mathcal{A}=\{M~|~\Ext_{R}^{1}(M, A)=0 ~\text{for all}~ A\in \mathcal{A}\}$. A pair ($\mathcal{A}, \mathcal{B}$) of classes of left $R$-modules is called a cotorsion pair if $\mathcal{A}^{\bot}=\mathcal{B}$ and $^{\bot}\mathcal{B}=\mathcal{A}$.

 Let $\mathcal{A}$ be a class of left $R$-modules. A morphism $f: A\rightarrow M$ is called an $\mathcal{A}$-precover if $A\in \mathcal{A}$ and the induced morphism $\Hom_{R}(A', f): \Hom_{R}(A', A)\rightarrow \Hom_{R}(A', M)$ is an epimorphism for each $A'\in \mathcal{A}$. An $\mathcal{A}$-precover $f: A\rightarrow M$ is called an $\mathcal{A}$-cover if every endomorphism $g: A\rightarrow A$ such that $fg=f$ is an automorphism of $A$. The class $\mathcal{A}$ is called precovering (resp. covering) if every module $M$ has an $\mathcal{A}$-precover (resp. $\mathcal{A}$-cover). Dually we have the definitions of an $\mathcal{A}$-(pre)envelope and an $\mathcal{A}$-(pre)enveloping class.

 A class $\mathcal{A}$ of left $R$-modules is projectively resolving if it is closed under extensions and kernels of epimorphisms, and it contains the class of projective left $R$-modules. Dually we have the notion of injectively coresolving class.

 A cotorsion pair ($\mathcal{A}, \mathcal{B}$) is said to be hereditary if $\mathcal{A}$ is projectively resolving, or equivalently, if $\mathcal{B}$ is injectively coresolving. A cotorsion pair ($\mathcal{A}, \mathcal{B}$) is said to be complete if for any module $M$ there exist exact sequences $0\rightarrow M\rightarrow B\rightarrow A\rightarrow 0$ and $0\rightarrow B'\rightarrow A'\rightarrow M\rightarrow 0$ with $A, A'\in \mathcal{A}$ and $B, B'\in \mathcal{B}$. A cotorsion pair ($\mathcal{A}, \mathcal{B}$) is said to be perfect if every module has an $\mathcal{A}$-cover and a $\mathcal{B}$-envelope. It is well known that a perfect cotorsion pair is complete, but the converse may be false in general.

\textbf{Duality pairs.} In \cite[Definition 2.1]{HJ09}, a duality pair over a ring $R$ is a pair ($\mathcal{X}, \mathcal{Y}$), of classes of $R$-modules, satisfying (1) $X\in \mathcal{X}$ if and only if $X^{+}\in \mathcal{Y}$, and (2) $\mathcal{Y}$ is closed under direct summands and finite direct sums, where $X^{+}=\Hom_{\mathbb{Z}}(M, \mathbb{Q}/\mathbb{Z})$. A duality pair ($\mathcal{X}, \mathcal{Y}$) is called perfect if $\mathcal{X}$ contains the module $R$, and is closed under direct sums and extensions.

\begin{lemma}(\cite[Theorem 3.1]{HJ09})\label{lem1.1} Let $(\mathcal{X},\mathcal{Y})$ be a duality pair. Then the following hold:

$\mathrm{(1)}$ $\mathcal{X}$ is closed under pure submodules, pure quotients, and pure extensions.

$\mathrm{(2)}$ If $(\mathcal{X},\mathcal{Y})$ is perfect, then $(\mathcal{X},\mathcal{X}^{\bot})$ is a perfect cotorsion pair.
\end{lemma}

\begin{lemma}(\cite[Proposition 2.3]{Gil17})\label{lem1.2} If $(\mathcal{X},\mathcal{Y})$ is a perfect duality pair, then $\mathcal{X}$ contains all projective modules, in fact, it contains all flat modules. And, the class $\mathcal{Y}$ contains all injective modules.
\end{lemma}

Recall that in \cite[Definition 2.4]{Gil17} a duality pair $(\mathcal{X},\mathcal{Y})$ is called symmetric if both $(\mathcal{X},\mathcal{Y})$ and $(\mathcal{Y},\mathcal{X})$ are duality pairs. Several examples of perfect and symmetric duality pairs are given in  \cite{ BP17, EIP17, Gil17, HJ09}.

 \begin{example}

 (1) $(\mathcal{F}_{n},\mathcal{I}_{n})$ is a perfect duality pair, where $\mathcal{F}_{n}$ is the class of all modules $M$ with $\fd(M)\leqslant n$, $\mathcal{I}_{n}$ is the class of all modules $N$ with $\id(N)\leqslant n$. Over a left noetherian ring $R$, $(\mathcal{F}_{n},\mathcal{I}_{n})$ is a perfect and symmetric duality pair.

 (2) $(\mathcal{F}_{n},\mathcal{FI}_{n})$ is a perfect duality pair, where $\mathcal{FI}_{n}$ is the class of all modules $N$ with FP-$\id(N)\leqslant n$. Over a right coherent ring $R$, $(\mathcal{F}_{n},\mathcal{FI}_{n})$ is a perfect and symmetric duality pair.

 (3) In \cite{BP17}, a right $R$-module $F$ is said to be of type $FP_{n}$ if it has a projective resolution $P_{n}\rightarrow \cdots\rightarrow P_{2}\rightarrow P_{1}\rightarrow P_{0}\rightarrow F \rightarrow 0$ with each $P_{i}$ finitely generated. A right $R$-module $M$ is called $FP_n$-injective if $\Ext_{R}^{1}(F, M)=0$ for all $R$-modules $F$ of type $FP_{n}$, a left $R$-module $N$ is called $FP_n$-flat if $\Tor^{R}_{1}(F, N)=0$ for all $R$-modules $F$ of type $FP_{n}$. Let $\mathcal{FP}_{n}$-Flat denote the class of all $FP_n$-flat left $R$-modules, and $\mathcal{FP}_{n}$-Inj denote the class of all $FP_n$-injective right $R$-modules. Then for all $n\geqslant2$,($\mathcal{FP}_{n}$-Flat, $\mathcal{FP}_{n}$-Inj) is a perfect and symmetric duality pair by \cite[Corollary 3.7]{BP17}.

 (4) Recall that in \cite{M07} a ring $R$ is said to be right min-coherent in case every simple right ideal is finitely presented. A left $R$-module $M$ is called min-flat if $\Tor_{1}^{R}(R/I, M)=0$ for any simple right ideal $I$. A right $R$-module $M$ is called min-injective if $\Ext^{1}_{R}(R/I, M)=0$ for any simple right ideal $I$.  Let $\mathcal{MF}$ denote the class of all min-flat left $R$-modules, and $\mathcal{MI}$ denote the class of all min-injective right $R$-modules. Then ($\mathcal{MF}$, $\mathcal{MI}$) is a perfect duality pair by \cite[Lemma 3.2]{M07}. Over a right min-coherent ring $R$, ($\mathcal{MF}$, $\mathcal{MI}$) is a perfect and symmetric duality pair by \cite[Theorem  4.5]{M07}.
\end{example}
Throughout this paper, $R$ is an associative ring with an identity, the modules are unital. $R$-Mod denotes the category of all left $R$-modules. Let $\mathcal{P}$ and $\mathcal{F}$ be the classes of all projective and flat left $R$-modules, respectively, and $\mathcal{I}$ be the class of all injective right $R$-modules. For an $R$-module $M$, $M^{+}$ denotes the character module $\Hom_{\mathbb{Z}}(M, \mathbb{Q}/\mathbb{Z})$.

\section{Gorenstein $(\mathcal{X}, \mathcal{Y})$-flat modules}
Let $\mathcal{X}$ be a class of left $R$-modules, $\mathcal{Y}$ be a class of right $R$-modules. In the section, we introduce and study Gorenstein $(\mathcal{X}, \mathcal{Y})$-flat modules, and show that the class of all Gorenstein $(\mathcal{X}, \mathcal{Y})$-flat modules have a strong stability. In particular, when $(\mathcal{X}, \mathcal{Y})$ is a perfect (symmetric) duality pair, we give some functorial descriptions of Gorenstein $(\mathcal{X}, \mathcal{Y})$-flat dimension.

\begin{definition} A left $R$-module $M$ is called Gorenstein $(\mathcal{X}, \mathcal{Y})$-flat, if there exists an exact sequence of left $R$-modules in $\mathcal{X}$,
$$\cdots\rightarrow X_{1}\rightarrow X_{0}\rightarrow X^{0}\rightarrow X^{1}\rightarrow\cdots$$
such that $M\cong \Ker(X^{0}\rightarrow X^{1})$ and $Y\otimes_{R}-$ leaves the sequence exact whenever $Y$ in $\mathcal{Y}$.
\end{definition}

Use $\mathcal{GF}_{\mathcal{(X,Y)}}(R)$ to denote the class of all Gorenstein $(\mathcal{X}, \mathcal{Y})$-flat modules.

\begin{remark}\label{rem2.1}  (1) It is clear that each module in $\mathcal{X}$ is in $\mathcal{GF}_{\mathcal{(X,Y)}}(R)$.

(2) If
$$\mathbb{X}: \cdots\rightarrow X_{1}\rightarrow X_{0}\rightarrow X^{0}\rightarrow X^{1}\rightarrow\cdots$$
is a $\mathcal{Y}\otimes_{R}-$ exact exact sequence of modules in $\mathcal{X}$, then by symmetic, all the kernels, the images, and the cokernels of $\mathbb{X}$ are in $\mathcal{GF}_{\mathcal{(X,Y)}}(R)$.
\end{remark}

\begin{example}\label{ex2.1}

(1) Obviously, if $\mathcal{X}=\mathcal{F}$, $\mathcal{Y}=\mathcal{I}$, then Gorenstein $(\mathcal{X}, \mathcal{Y})$-flat modules are exactly Gorenstein flat modules in \cite{EJT93}. If $\mathcal{X}=\mathcal{F}$, $\mathcal{Y}$ is the class of all FP-injective modules, then Gorenstein $(\mathcal{X}, \mathcal{Y})$-flat modules are also Gorenstein flat modules by \cite[Lemma 2.8]{MD08}.

(2) If $\mathcal{X}=\mathcal{F}$, then Gorenstein $(\mathcal{X}, \mathcal{Y})$-flat modules are exactly Gorenstein $\mathcal{Y}$-flat modules $\mathcal{GF}_{\mathcal{Y}}(R)$ in \cite{EIP17}. If $\mathcal{Y}$ is the class of all absolutely clean modules, then Gorenstein $\mathcal{Y}$-flat modules are precisely Gorenstein AC-flat modules in \cite{BEI17}.

(3)  Let $\mathcal{F}^{n}=\{M~|~\Tor_{1}^{R}(N,~M)=0~\text{for all finitely presented modules}~ N ~\text{with} ~\pd N\leqslant n \}$, $\mathcal{FI}^{n}=\{M~|~\Ext^{1}_{R}(N,~M)=0~\text{for all finitely presented modules}~ N ~\text{with} ~\pd N\leqslant n \}$. Then Gorenstein $(\mathcal{F}^{n}, \mathcal{FI}^{n})$-flat modules are exactly Gorenstein $n$-flat modules in \cite{SUU14}.

(4) Let $R$ be a commutative ring, $C$ be a semiduality module, and let $\mathcal{F}_{C}=\{C\otimes_{R}F~|~F ~\text{is a flat module}\}$, $\mathcal{I}_{C}=\{\Hom_{R}(C, I)~|~I~ \text{is an injective module}\}$. Then Gorenstein $(\mathcal{F}_{C}, \mathcal{I}_{C})$-flat modules coincide with Gorenstein $C$-flat modules in \cite{WGM17}. If $R$ is a Cohen-Macaulay ring of finite Krull dimension admitting a dualizing module $\Omega$, then Gorenstein $(\mathcal{F}_{\Omega}, \mathcal{I}_{\Omega})$-flat modules are precisely $\Omega$-Gorenstein flat modules in \cite[Definition 12.5.14]{EJ00}.
\end{example}

\begin{remark}\label{rem2.2} (1) If $\mathcal{X}$ is closed under direct sums, then
$\mathcal{GF}_{\mathcal{(X,Y)}}(R)$ is closed under direct sums.

(2) If $\mathcal{X}_{1}\subseteq\mathcal{X}_{2}$, then $\mathcal{GF}_{(\mathcal{X}_{1},\mathcal{Y})}(R)\subseteq
\mathcal{GF}_{(\mathcal{X}_{2},\mathcal{Y})}(R)$. If $\mathcal{Y}_{1}\subseteq\mathcal{Y}_{2}$, then $\mathcal{GF}_{(\mathcal{X},\mathcal{Y}_{2})}(R)
\subseteq\mathcal{GF}_{(\mathcal{X},\mathcal{Y}_{1})}(R)$.
\end{remark}

Recall that an exact sequence $\cdots\rightarrow X_{1}\rightarrow X_{0}\rightarrow M\rightarrow 0$ in $R$-Mod with each $X_{i}\in \mathcal{X}$ is said to be an $\mathcal{X}$-resolution of $M$. An $\mathcal{X}$-coresolution of $M$ is defined dually. An exact sequence in $R$-Mod is called $\mathcal{Y}\otimes_{R}-$ exact if it remains still exact after applying the functor $Y\otimes_{R}-$ for all $Y\in \mathcal{Y}$.

\begin{lemma}\label{lem2.1} Assume that $\mathcal{GF}_{\mathcal{(X,Y)}}(R)$ is closed under extensions. Then a left $R$-module $M$ has a $\mathcal{Y}\otimes_{R}-$ exact $\mathcal{X}$-resolution if and only if it has a $\mathcal{Y}\otimes_{R}-$ exact $\mathcal{GF}_{\mathcal{(X,Y)}}(R)$-resolution.
\end{lemma}

\begin{proof}It is enough to show that the $\lq\lq$if~" part. Let $0\rightarrow N\rightarrow G_{0}\rightarrow M\rightarrow 0$ be an $\mathcal {Y}\otimes_{R}-$ exact exact sequence with $G_{0} \in \mathcal{GF}_{\mathcal{(X,Y)}}(R)$ and $N$ having a $\mathcal{Y}\otimes_{R}-$ exact $\mathcal{GF}_{\mathcal{(X,Y)}}(R)$-resolution. Then we have the following pullback diagram
 $$\xymatrix{
            &  0 \ar[d]_{ }                            & 0   \ar[d]_{ }          &               &  \\
&             G^{'} \ar[d]_{ }  \ar@{=}[r]^{}             &G^{'} \ar[d]_{}&                &  \\
0 \ar[r]^{} & H \ar[d]_{ }  \ar[r]^{ } &X_{0}  \ar[d]_{ } \ar[r]^{ } & M  \ar[r]^{ }\ar@{=}[d]^{}  &0 \\
0 \ar[r]^{} & N  \ar[d]_{ }           \ar[r]^{ } &G_{0}\ar[d]_{} \ar[r]^{} & M \ar[r]^{ } &0 \\
            &0                            &0             & &   }$$
with~$X_{0}\in \mathcal {X}$, $G^{'}\in \mathcal{GF}_{\mathcal{(X,Y)}}(R)$. Since the bottom row and the middle column are $\mathcal {Y}\otimes_{R}-$ exact, so is the middle row by Snake lemma. Note that there is an $\mathcal {Y}\otimes_{R}-$ exact exact sequence $0\rightarrow K\rightarrow G_{1}\rightarrow N\rightarrow 0 $\ such that $K$ has a $\mathcal{Y}\otimes_{R}-$ exact $\mathcal{GF}_{\mathcal{(X,Y)}}(R)$-resolution and $G_{1} \in \mathcal{GF}_{\mathcal{(X,Y)}}(R)$. Consider the following pullback diagram
$$\xymatrix{
            &                            & 0   \ar[d]_{ }          & 0 \ar[d]_{ }                &  \\
&                            &K \ar[d]_{} \ar@{=}[r]^{}& K \ar[d]_{ }                &  \\
0 \ar[r]^{} & G^{'} \ar@{=}[d]  \ar[r]^{ } &L  \ar[d]_{ } \ar[r]^{ } & G_{1}  \ar[d]_{ } \ar[r]^{ } &0 \\
0 \ar[r]^{} & G^{'}             \ar[r]^{ } &H\ar[d]_{} \ar[r]^{} & N \ar[d]_{ } \ar[r]^{ } &0 \\
            &                            &0             & 0               &   }$$
Since $\mathcal{GF}_{\mathcal{(X,Y)}}(R)$ is closed under extensions, $L\in \mathcal{GF}_{\mathcal{(X,Y)}}(R)$. Thus, $H$ has a $\mathcal{Y}\otimes_{R}-$ exact $\mathcal{GF}_{\mathcal{(X,Y)}}(R)$-resolution. Note that $0\rightarrow H\rightarrow X_{0}\rightarrow M\rightarrow 0$\ is $\mathcal {Y}\otimes_{R}-$ exact. By repeating the preceding process, we get that $M$ has a $\mathcal{Y}\otimes_{R}-$ exact $\mathcal {X}$-resolution.
\end{proof}

Dually, we can prove the following lemma.

\begin{lemma}\label{lem2.2} Assume that $\mathcal{GF}_{\mathcal{(X,Y)}}(R)$ is closed under extensions. Then a left $R$-module $M$ has a $\mathcal{Y}\otimes_{R}-$ exact $\mathcal{X}$-coresolution if and only if it has a $\mathcal{Y}\otimes_{R}-$ exact $\mathcal{GF}_{\mathcal{(X,Y)}}(R)$-coresolution.
\end{lemma}

Let $\mathcal{GF}^{2}_{\mathcal{(X,Y)}}(R)$ be the class of left $R$-modules $M$ for which there exists an exact sequence of left $R$-modules in $\mathcal{GF}_{\mathcal{(X,Y)}}(R)$
$$\cdots \rightarrow G_{1}\rightarrow G_{0} \rightarrow G^{0} \rightarrow G^{1}\rightarrow \cdots
$$
such that $M\cong \Ker(G^{0}\rightarrow G^{1})$ and such that $\mathcal{Y}\otimes_{R}-$ leaves the sequence exact.
It is obvious that $\mathcal{GF}_{\mathcal{(X,Y)}}(R)\subseteq
\mathcal{GF}^{2}_{\mathcal{(X,Y)}}(R)$. As a consequence of Lemmas \ref{lem2.1} and \ref{lem2.2}, we have the following result.

\begin{theorem}\label{the2.1} If $\mathcal{GF}_{\mathcal{(X,Y)}}(R)$ is closed under extensions, then $\mathcal{GF}_{\mathcal{(X,Y)}}(R)=
\mathcal{GF}^{2}_{\mathcal{(X,Y)}}(R)$.
\end{theorem}

Note that Theorem \ref{the2.1} recovers \cite[Main Theorem]{BK12}, \cite[Theorem 3.12]{WGM17}, \cite[Theorem 2.13]{XD13}, and \cite[Theorem 4.3]{YL12-1}.

\begin{corollary}\label{cor2.1} Assume that $\mathcal{GF}_{\mathcal{(X,Y)}}(R)$ is closed under extensions. Then the following are equivalent for left a $R$-module $M$:

$\mathrm{(1)}$ $M$ is in $\mathcal{GF}_{\mathcal{(X,Y)}}(R)$.

$\mathrm{(2)}$ There exists a $\mathcal{Y}\otimes_{R}-$ exact exact sequence of left $R$-modules in $\mathcal{U}$
$$\cdots \rightarrow G_{1}\rightarrow G_{0} \rightarrow G^{0} \rightarrow G^{1}\rightarrow \cdots
$$
and $M\cong \Ker(G^{0}\rightarrow G^{1})$ for some class $\mathcal{U}$ with $\mathcal{X}\subseteq\mathcal{U}\subseteq \mathcal{GF}_{\mathcal{(X,Y)}}(R)$.

$\mathrm{(3)}$ There exists a $\mathcal{V}\otimes_{R}-$ exact exact sequence of left $R$-modules in $\mathcal{GF}_{\mathcal{(X,Y)}}(R)$
$$\cdots \rightarrow G_{1}\rightarrow G_{0} \rightarrow G^{0} \rightarrow G^{1}\rightarrow \cdots
$$
and $M\cong \Ker(G^{0}\rightarrow G^{1})$ for some class $\mathcal{V}$ with $\mathcal{V}\subseteq\mathcal{Y}$.
\end{corollary}

Applying Corollary \ref{cor2.1} to Example \ref{ex2.1}(2), we have

\begin{corollary} \label{cor2.2} Assume that $\mathcal{GF}_{\mathcal{Y}}(R)$ is closed under extensions. Then the following are equivalent for a left $R$-module $M$:

$\mathrm{(1)}$ $M$ is in $\mathcal{GF}_{\mathcal{Y}}(R)$.

$\mathrm{(2)}$ There exists a $\mathcal{Y}\otimes_{R}-$ exact exact sequence of left $R$-modules in $\mathcal{U}$
$$\cdots \rightarrow G_{1}\rightarrow G_{0} \rightarrow G^{0} \rightarrow G^{1}\rightarrow \cdots
$$
and $M\cong \Ker(G^{0}\rightarrow G^{1})$ for some class $\mathcal{U}$ with $\mathcal{F}\subseteq\mathcal{U}\subseteq \mathcal{GF}_{\mathcal{Y}}(R)$.

$\mathrm{(3)}$ There exists a $\mathcal{V}\otimes_{R}-$ exact exact sequence of left $R$-modules in $\mathcal{GF}_{\mathcal{Y}}(R)$
$$\cdots \rightarrow G_{1}\rightarrow G_{0} \rightarrow G^{0} \rightarrow G^{1}\rightarrow \cdots
$$
and $M\cong \Ker(G^{0}\rightarrow G^{1})$ for some class $\mathcal{V}$ with $\mathcal{V}\subseteq\mathcal{Y}$.
\end{corollary}

Recall that in \cite[Definition 2.6]{BGH14} a left $R$-module $M$ is called absolutely clean or FP$_{\infty}$-injective if $\Ext_{R}^{1}(N, M)=0$ for all left $R$-modules $N$ of type FP$_{\infty}$. If $\mathcal{Y}$ is the class of all absolutely clean $R$-modules, then Gorenstein $\mathcal{Y}$-flat modules are precisely Gorenstein AC-flat modules in \cite[Definition 4.1]{BEI17}. Corollary \ref{cor2.2} shows that Gorenstein AC-flat modules have a strong stability.

We say that $\mathcal{Y}\top\mathcal{X}$ if $\Tor_{i}^{R}(Y, X)=0$ for all $Y\in \mathcal{Y}$, $X\in \mathcal{X}$ and $i\geqslant1$. In the following, we always assume that $\mathcal{Y}\top\mathcal{X}$, and $\mathcal{P}\subseteq\mathcal{X}$.

\begin{lemma}\label{lem2.3} The following assertions are equivalent for a left $R$-module $M$.

$\mathrm{(1)}$ $M$ is Gorenstein $\mathcal{(X,Y)}$-flat.

$\mathrm{(2)}$ $\Tor_{i}^{R}(Y, M)=0$ for all $i>0$ and all $Y\in \mathcal{Y}$, and there exists an exact sequence of left $R$-modules $0\rightarrow M\rightarrow X^{0}\rightarrow X^{1}\rightarrow \cdots$ with each $X^{i}\in \mathcal{X}$, such that $\mathcal{Y}\otimes_{R}-$ leaves the sequence exact.

$\mathrm{(3)}$ There exists a short exact sequence of left $R$-modules $0\rightarrow M\rightarrow X\rightarrow G\rightarrow 0$ with $X\in \mathcal{X}$ and $G\in \mathcal{GF}_{\mathcal{(X,Y)}}(R)$.
\end{lemma}

\begin{proof} Using the definition of Gorenstein $\mathcal{(X,Y)}$-flat modules, it is obtained by standard argument similar to the proof of \cite[Lemma 2.4]{Ben09}.
\end{proof}

\begin{lemma}\label{lem2.4} Let $0\rightarrow M\rightarrow L\rightarrow N\rightarrow 0$ be a short exact sequence in $R$-Mod. If $\mathcal{X}$ is closed under extensions, $M\in \mathcal{GF}_{\mathcal{(X,Y)}}(R)$, and $N \in \mathcal{X}$, then $L\in \mathcal{GF}_{\mathcal{(X,Y)}}(R)$.
\end{lemma}

\begin{proof}Since $M\in \mathcal{GF}_{\mathcal{(X,Y)}}(R)$, there exists an exact sequence of left $R$-modules $0\rightarrow M\rightarrow X\rightarrow G\rightarrow 0$, where $X \in \mathcal{X}$, $G\in \mathcal{GF}_{\mathcal{(X,Y)}}(R)$. Consider the following pushout diagram
\begin{center}
$\xymatrix{
     &&  0\ar[d]_{}  & 0 \ar[d]_{} &  \\
      & 0 \ar[r]^{} & M \ar[d]_{} \ar[r]^{} & L\ar[d]_{} \ar[r]^{} & N \ar@{=}[d]_{} \ar[r]^{} & 0  \\
      &0  \ar[r]^{} & X  \ar[d]\ar[r]^{} &U\ar[d]_{} \ar[r]^{} & N  \ar[r]^{} & 0  \\
      && G\ar@{=}[r]^{} \ar[d]& G\ar[d] & \\
    & & 0& 0  &
      }$

\end{center}
Because both $X$ and $N$ are in $\mathcal{X}$, $U$ is in $\mathcal{X}$. Then, by the middle column and from Lemma \ref{lem2.3}, $L$ is in $\mathcal{GF}_{\mathcal{(X,Y)}}(R)$.
\end{proof}

\begin{proposition}\label{prop2.1} The following conditions are equivalent for a ring $R$:

$\mathrm{(1)}$ The class $\mathcal{GF}_{\mathcal{(X,Y)}}(R)$ is closed under extensions.

$\mathrm{(2)}$ The class $\mathcal{GF}_{\mathcal{(X,Y)}}(R)$ is projectively resolving.
\end{proposition}
\begin{proof}(2)$\Rightarrow$ (1) is clear.

(1)$\Rightarrow$ (2) To claim that the class $\mathcal{GF}_{\mathcal{(X,Y)}}(R)$ is projectively resolving, it suffices to prove that it is closed under kernels of epimorphisms. Then, consider a short exact sequence of left $R$-modules $0\rightarrow M\rightarrow L\rightarrow N\rightarrow 0$, where $L$ and $N$ are in $\mathcal{GF}_{\mathcal{(X,Y)}}(R)$. We show that $M$ is in $\mathcal{GF}_{\mathcal{(X,Y)}}(R)$. Since $L$ is in $\mathcal{GF}_{\mathcal{(X,Y)}}(R)$, there exists a short exact sequence of left $R$-modules $0\rightarrow L\rightarrow X\rightarrow G\rightarrow 0$, where $X \in \mathcal{X}$, $G\in \mathcal{GF}_{\mathcal{(X,Y)}}(R)$. Consider the following pushout diagram
\begin{center}
$\xymatrix{
      & & 0\ar[d]_{}  & 0 \ar[d]_{} &  \\
       0\ar[r]& M\ar@{=}[d]^{} \ar[r]& L\ar[d]\ar[r] &N\ar[d]\ar[r]&0 \\
 0\ar[r]& M \ar[r]& X\ar[d]\ar[r] &U\ar[d]\ar[r]&0 \\
&&G\ar[d]_{} \ar@{=}[r]^{} & G \ar[d]_{} \\
    & & 0& 0  &
      }$
\end{center}
By the right column and (1), the $R$-module $U$ is in $\mathcal{GF}_{\mathcal{(X,Y)}}(R)$. Therefore, by the middle row and Lemma \ref{lem2.3}, $M$ is in $\mathcal{GF}_{\mathcal{(X,Y)}}(R)$.
\end{proof}

\begin{corollary}\label{cor2.3} Let $M$ be a left $R$-module. Consider two exact sequences of left $R$-modules,

$$0\rightarrow G_{n}\rightarrow G_{n-1}\rightarrow \cdots\rightarrow G_{0}\rightarrow M\rightarrow 0,~~\text{and}$$
$$0\rightarrow H_{n}\rightarrow H_{n-1}\rightarrow \cdots\rightarrow H_{0}\rightarrow M\rightarrow 0,$$
where $G_{0}, \cdots, G_{n-1}$ and $H_{0}, \cdots, H_{n-1}$ are in $\mathcal{GF}_{\mathcal{(X,Y)}}(R)$. Assume that the class $\mathcal{GF}_{\mathcal{(X,Y)}}(R)$ is closed under extensions, then $G_{n}\in \mathcal{GF}_{\mathcal{(X,Y)}}(R)$ if and only if $H_{n}\in \mathcal{GF}_{\mathcal{(X,Y)}}(R)$.
\end{corollary}

\begin{proof}It is obtained by Proposition \ref{prop2.1} and \cite[Lemma 2.1]{zhu13}.
\end{proof}

From \cite[Proposition 1.4]{Hol04}, Remark \ref{rem2.2} and Proposition \ref{prop2.1}, we have the following corollary.

\begin{corollary}If the class $\mathcal{GF}_{\mathcal{(X,Y)}}(R)$ is closed under extensions, then the class $\mathcal{GF}_{\mathcal{(X,Y)}}(R)$ is closed under direct summands.
\end{corollary}

\begin{proposition}\label{prop2.2} Assume that $(\mathcal{X},\mathcal{Y})$ is a perfect duality pair. Then the following conditions are equivalent for a ring $R$:

$\mathrm{(1)}$ The class $\mathcal{GF}_{\mathcal{(X,Y)}}(R)$ is closed under extensions.

$\mathrm{(2)}$ For every short exact sequence of left $R$-modulles $0\rightarrow G_{1}\rightarrow G_{0}\rightarrow M\rightarrow 0$, where $G_{0}, G_{1}\in \mathcal{GF}_{\mathcal{(X,Y)}}(R)$. If $\Tor_{1}^{R}(Y, M)=0$ for all $Y\in \mathcal{Y}$, then $M\in \mathcal{GF}_{\mathcal{(X,Y)}}(R)$.

\end{proposition}
\begin{proof}(1)$\Rightarrow$ (2) Since $G_{1}$ is in $\mathcal{GF}_{\mathcal{(X,Y)}}(R)$, there exists a short exact sequence $0\rightarrow G_{1}\rightarrow X_{1}\rightarrow H\rightarrow 0$, where $X_{1} \in \mathcal{X}$, $G_{1}\in \mathcal{GF}_{\mathcal{(X,Y)}}(R)$. Consider the following pushout diagram
\begin{center}
$\xymatrix{&  0 \ar[d]_{ }  & 0   \ar[d]_{ }  &     &  \\
0 \ar[r]^{} & G_{1} \ar[d]_{ }  \ar[r]^{ } &G_{0}  \ar[d]_{ } \ar[r]^{ } & M  \ar[r]^{ }\ar@{=}[d]^{}  &0 \\
0 \ar[r]^{} & X_{1}  \ar[d]_{ }\ar[r]^{ } &U\ar[d]_{} \ar[r]^{} & M \ar[r]^{ } &0 \\
&   H \ar[d]_{ }  \ar@{=}[r]^{}             &H \ar[d]_{}&                &  \\
 &0              &0             & &   }$
\end{center}
In the middle column both $G_{0}$ and $H$ are in $\mathcal{GF}_{\mathcal{(X,Y)}}(R)$, then so is $U$ by (1). Hence, there exists a short exact sequence $0\rightarrow U\rightarrow X\rightarrow G\rightarrow 0$, where $X \in \mathcal{X}$, $G\in \mathcal{GF}_{\mathcal{(X,Y)}}(R)$. Consider the following pushout diagram
$$\xymatrix{&& 0   \ar[d]_{ }  & 0 \ar[d]_{ }    &  \\
0 \ar[r]^{} &X_{1} \ar@{=}[d]  \ar[r]^{ } &U  \ar[d]_{ } \ar[r]^{ }
& M\ar[d]_{ } \ar[r]^{ } &0 \\
0 \ar[r]^{} &X_{1}\ar[r]^{ } &X\ar[d]_{} \ar[r]^{} & V \ar[d]_{ } \ar[r]^{ } &0 \\
& &G \ar[d]_{} \ar@{=}[r]^{}& G \ar[d]_{ } &  \\
&       &0   & 0   &   }$$
We prove that $V$ is in $\mathcal{X}$. Consider the right column $0\rightarrow M\rightarrow V\rightarrow G\rightarrow 0$. Since $G$ is in $\mathcal{GF}_{\mathcal{(X,Y)}}(R)$, we get that $\Tor_{1}^{R}(Y, V)=0$ for all $Y\in \mathcal{Y}$ by Lemma \ref{lem2.3}. On the other hand, by the middle row $0\rightarrow X_{1}\rightarrow X\rightarrow V\rightarrow 0$, we have the short exact sequence $0\rightarrow V^{+}\rightarrow X^{+}\rightarrow X_{1}^{+}\rightarrow 0$. By assumption, $X^{+}, X_{1}^{+}\in \mathcal{Y}$. Note that $\Ext^{1}_{R}(X_{1}^{+}, V^{+})\cong(\Tor_{1}^{R}(X_{1}^{+}, V))^{+}$. Hence, $\Ext^{1}_{R}(X_{1}^{+}, V^{+})=0$, and so the above sequence splits. Then $V^{+}$ is in $\mathcal{Y}$, which implies that $V$ is in
$\mathcal{X}$. Now, by Lemma \ref{lem2.3} and the short exact sequence
$0\rightarrow M\rightarrow V\rightarrow G\rightarrow 0$, we get that $M$ is in $\mathcal{GF}_{\mathcal{(X,Y)}}(R)$.

(2)$\Rightarrow$ (1) Let $0\rightarrow L\rightarrow M\rightarrow N\rightarrow 0$ be a short exact sequence, where $L,N\in \mathcal{GF}_{\mathcal{(X,Y)}}(R)$. We show that $M$ is in $\mathcal{GF}_{\mathcal{(X,Y)}}(R)$. Since $L$ and $N$ are in $\mathcal{GF}_{\mathcal{(X,Y)}}(R)$, we get that $\Tor_{i}^{R}(Y, M)=0$ for all $i>0$ and all $Y\in \mathcal{Y}$ by Lemma \ref{lem2.3}. On the other hand, since $N$ is in $\mathcal{GF}_{\mathcal{(X,Y)}}(R)$, there exists a short exact sequence $0\rightarrow G\rightarrow X\rightarrow N\rightarrow 0$, where $X\in \mathcal{X}$ and $G\in \mathcal{GF}_{\mathcal{(X,Y)}}(R)$. Consider the following pullback diagram
$$\xymatrix{& & 0   \ar[d]_{ }& 0 \ar[d]_{ }  &  \\
&   &G \ar[d]_{} \ar@{=}[r]^{}& G \ar[d]_{ }  &  \\
0 \ar[r]^{} & L \ar@{=}[d]  \ar[r]^{ } &W \ar[d]_{ } \ar[r]^{ } & X \ar[d]_{ } \ar[r]^{ } &0 \\
0 \ar[r]^{} & L \ar[r]^{ } &M\ar[d]_{} \ar[r]^{} & N \ar[d]_{ } \ar[r]^{ } &0 \\
 & &0       & 0 &   }$$
Since $L$ is in $\mathcal{GF}_{\mathcal{(X,Y)}}(R)$, there exists a short exact sequence $0\rightarrow L\rightarrow X'\rightarrow G'\rightarrow 0$, where $X'\in \mathcal{X}$ and $G'\in \mathcal{GF}_{\mathcal{(X,Y)}}(R)$. Consider the following pushout diagram
$$\xymatrix{&  0 \ar[d]_{ }& 0   \ar[d]_{ }   &      &  \\
0 \ar[r]^{} & L \ar[d]_{ }  \ar[r]^{ } &W  \ar[d]_{ } \ar[r]^{ } & X  \ar[r]^{ }\ar@{=}[d]^{}  &0 \\
0 \ar[r]^{} & X'  \ar[d]_{ }  \ar[r]^{ } &U\ar[d]_{} \ar[r]^{} & X \ar[r]^{ } &0 \\
&    G' \ar[d]_{ }  \ar@{=}[r]^{}    &G' \ar[d]_{}&       &  \\
 &0    &0     & &   }$$
 In the middle row both $X'$ and $X$ are in $\mathcal{X}$, then so is $U$. Hence, by the middle column and Lemma \ref{lem2.3}, $W$ is in $\mathcal{GF}_{\mathcal{(X,Y)}}(R)$. Now, in the short exact sequence $0\rightarrow G\rightarrow W\rightarrow M\rightarrow 0$, we note that $G$ and $W$ are in $\mathcal{GF}_{\mathcal{(X,Y)}}(R)$, and $\Tor_{i}^{R}(Y, M)=0$ for all $i>0$ and all $Y\in \mathcal{Y}$, then $M$ is in $\mathcal{GF}_{\mathcal{(X,Y)}}(R)$ by (2).
\end{proof}

Let $\mathcal{A}$ be a class of left $R$-modules, $M$ be a left $R$-module.  For a positive integer $n$, we say that $M$ has $\mathcal{A}$-projective dimension at most $n$, if $M$ has an $\mathcal{A}$-resolution of length $n$, and we write $\mathcal{A}$-$pd(M)\leqslant n$. If no such finite sequence exists, define $\mathcal{A}$-$pd(M)=\infty$. If $\mathcal{A}=\mathcal{GF}_{\mathcal{(X,Y)}}(R)$, then we call it Gorenstein $(\mathcal{X},\mathcal{Y})$-flat dimension of $M$. Dually, we have the notion of $\mathcal{A}$-injective dimension $\mathcal{A}$-$id(M)$.

Recall that in \cite{SWW08} a subclass $\mathcal{U}$ of $\mathcal{X}$ is called a generator (resp. cogenerator) for $\mathcal{X}$ if for each module in $\mathcal{X}$, there exists an exact sequence $0\rightarrow X'\rightarrow U\rightarrow X\rightarrow 0$ (resp. $0\rightarrow X\rightarrow U\rightarrow X'\rightarrow 0$) in $R$-Mod with $X'$ in $\mathcal{X}$ and $U$ in $\mathcal{U}$. By definitions, it is clear that $\mathcal{X}$ is a generator-cogenerator for $\mathcal{GF}_{\mathcal{(X,Y)}}(R)$.

\begin{theorem}\label{the2.2} Assume that $(\mathcal{X},\mathcal{Y})$ is a perfect duality pair, and $\mathcal{GF}_{\mathcal{(X,Y)}}(R)$ is closed under extensions. Then the following are equivalent for a module $M$ and a positive integer $n$:

$\mathrm{(1)}$ $\mathcal{GF}_{\mathcal{(X,Y)}}(R)$-$pd(M)\leqslant n$.

$\mathrm{(2)}$ $\mathcal{GF}_{\mathcal{(X,Y)}}(R)$-$pd(M)<\infty$ and $\Tor_{i}^{R}(Y, M)=0$ for all $i>n$ and each $Y\in \mathcal{Y}$.

$\mathrm{(3)}$ $\mathcal{GF}_{\mathcal{(X,Y)}}(R)$-$pd(M)<\infty$ and $\Tor_{i}^{R}(Z, M)=0$ for all $i>n$ and each $Z$ with $\mathcal{Y}$-$id(Z)<\infty$.

$\mathrm{(4)}$ For every exact sequence of left $R$-modules $0\rightarrow K_{n}\rightarrow G_{n-1}\rightarrow\cdots \rightarrow G_{1}\rightarrow G_{0}\rightarrow M\rightarrow 0$, if each $G_{i}$ is in $\mathcal{GF}_{\mathcal{(X,Y)}}(R)$, then so is $K_{n}$.

$\mathrm{(5)}$ There exists an exact sequence:
$$0\rightarrow G_{n}\rightarrow X_{n-1}\rightarrow\cdots \rightarrow X_{1}\rightarrow X_{0}\rightarrow M\rightarrow 0$$
with $G_{n}\in \mathcal{GF}_{\mathcal{(X,Y)}}(R)$ and all $X_{i}\in \mathcal{X}$.

$\mathrm{(6)}$ There exists an exact sequence:
$$0\rightarrow X_{n}\rightarrow X_{n-1}\rightarrow\cdots \rightarrow X_{1}\rightarrow G_{0}\rightarrow M\rightarrow 0$$
with $G_{0}\in \mathcal{GF}_{\mathcal{(X,Y)}}(R)$ and all $X_{i}\in \mathcal{X}$.

$\mathrm{(7)}$ For every non-negative integer $t$ such that $0\leqslant t\leqslant n$, there exists an exact sequence:
$$0\rightarrow X_{n}\rightarrow X_{n-1}\rightarrow\cdots \rightarrow X_{1}\rightarrow X_{0}\rightarrow M\rightarrow 0$$
with $X_{t}\in \mathcal{GF}_{\mathcal{(X,Y)}}(R)$ and all $X_{i}\in \mathcal{X}$ for $i\neq t$.

$\mathrm{(8)}$ For every non-negative integer $t$ such that $0\leqslant t\leqslant n$, there exists an exact sequence:
$$0\rightarrow G_{n}\rightarrow G_{n-1}\rightarrow\cdots \rightarrow G_{1}\rightarrow G_{0}\rightarrow M\rightarrow 0$$
with $G_{t}\in\mathcal{X}$ and all $G_{i}\in \mathcal{GF}_{\mathcal{(X,Y)}}(R)$ for $i\neq t$.
\end{theorem}
\begin{proof}By Lemma \ref{lem2.3}, Corollary \ref{cor2.3}, and Proposition \ref{prop2.2},  (1)$\Leftrightarrow$(2)$\Leftrightarrow$(3)$\Leftrightarrow$(4) are obtained by the similar proof of \cite[Theorem 2.8]{Ben09}. Since $\mathcal{X}$ is a generator-cogenerator for $\mathcal{GF}_{\mathcal{(X,Y)}}(R)$, by \cite[Theorem 5.5]{Hu13}, we get  (1)$\Leftrightarrow$(5)$\Leftrightarrow$(6)$\Leftrightarrow$(7)$\Leftrightarrow$(8).
\end{proof}
\begin{remark}Associated with \cite[Lemma 2.3]{EIP17}, a careful reading of proofs of Proposition \ref{prop2.2} and Theorem \ref{the2.2} give the corresponding results for $\mathcal{GF}_{\mathcal{Y}}(R)$.
\end{remark}

\begin{proposition}\label{prop2.4} Assume that $\mathcal{Y}$ is injectively coresolving, and $(\mathcal{X},\mathcal{Y})$ is a perfect duality pair. If a left $R$-module $M$ is Gorenstein $(\mathcal{X},\mathcal{Y})$-flat and $\mathcal{X}$-$\pd(M)<\infty$, then it is in $\mathcal{X}$.
\end{proposition}

\begin{proof}Since $M$ is Gorenstein $(\mathcal{X},\mathcal{Y})$-flat, there exists an exact sequence of left $R$-modules in $\mathcal{X}$,
$$\mathbb{X}: \cdots\rightarrow X_{1}\rightarrow X_{0}\rightarrow X^{0}\rightarrow X^{1}\rightarrow\cdots$$
such that $M\cong \Ker(X^{0}\rightarrow X^{1})$ and $Y\otimes_{R}\mathbb{X}$ is exact for each $Y\in \mathcal{Y}$. Then the exact sequence $\mathcal{Y}\otimes_{R}\mathbb{X}$ gives the exactness of $(\mathcal{Y}\otimes_{R}\mathbb{X})^{+}$, which yields the exact sequence $\Hom_{R}(\mathcal{Y}, \mathbb{X}^{+})$. Hence, by hypothesis, we have a $\Hom_{R}(\mathcal{Y}, -)$-exact exact sequence $0\rightarrow M^{+}\rightarrow X_{0}^{+}\rightarrow X_{1}^{+}\rightarrow \cdots\rightarrow X_{n-1}^{+}\rightarrow K_{n}^{+}\rightarrow 0$ with $X_{i}\in \mathcal{X}$ and $K_{n}^{+}\in \mathcal{Y}$. It is easy to see that $M^{+}\in \mathcal{Y}$, and so $M\in \mathcal{X}$.
\end{proof}

\begin{remark}Similar to the proof of Proposition \ref{prop2.4}, it is obtained the corresponding result for $\mathcal{GF}_{\mathcal{Y}}(R)$. That is, if $\mathcal{I}\subseteq\mathcal{Y}$, and $M$ is Gorenstein $\mathcal{Y}$-flat and $\fd(M)<\infty$, then $M$ is flat.
\end{remark}

\begin{proposition}\label{prop2.3} Assume that $(\mathcal{X},\mathcal{Y})$ is a perfect and symmetric duality pair, and $\mathcal{X}$ is closed under kernels of epimorphisms. Then the following conditions are equivalent for a ring $R$:

$\mathrm{(1)}$ $M$ is in $\mathcal{GF}_{\mathcal{(X,Y)}}(R)$.

$\mathrm{(2)}$ There exists an exact sequence of left $R$-modules in $\mathcal{X}$,
$$\cdots\rightarrow X_{1}\rightarrow X_{0}\rightarrow X^{0}\rightarrow X^{1}\rightarrow\cdots$$
such that $M\cong \Ker(X^{0}\rightarrow X^{1})$ and such that $\Hom_{R}(-, \mathcal{X}\cap \mathcal{X}^{\perp})$ leaves the sequence exact.

$\mathrm{(3)}$ $\Ext_{R}^{i}(M, W)=0$ for all $W\in \mathcal{X}\cap \mathcal{X}^{\perp}$ and there exists a $\Hom_{R}(-, \mathcal{X}\cap \mathcal{X}^{\perp})$-exact exact sequence
$$0\rightarrow M\rightarrow X^{0}\rightarrow X^{1}\rightarrow\cdots$$
with each $X^{i}\in \mathcal{X}$.

$\mathrm{(4)}$ $\Ext_{R}^{i}(M, W)=0$ for all $W\in \mathcal{X}\cap \mathcal{X}^{\perp}$ and there exists a $\Hom_{R}(-, \mathcal{X}\cap \mathcal{X}^{\perp})$-exact exact sequence
$$0\rightarrow M\rightarrow W^{0}\rightarrow W^{1}\rightarrow\cdots$$
with each $W^{i}\in \mathcal{X}\cap \mathcal{X}^{\perp}$.

$\mathrm{(5)}$ There exists a short exact sequence of left $R$-modules $0\rightarrow M\rightarrow W\rightarrow G\rightarrow 0$ where $W\in \mathcal{X}\cap \mathcal{X}^{\perp}$ and $G\in \mathcal{GF}_{\mathcal{(X,Y)}}(R)$.

\end{proposition}
\begin{proof}(1)$\Rightarrow$ (2) By (1), there exists an exact sequence of left $R$-modules in $\mathcal{X}$,
$$\cdots\rightarrow X_{1}\rightarrow X_{0}\rightarrow X^{0}\rightarrow X^{1}\rightarrow\cdots$$
such that $M\cong \Ker(X^{0}\rightarrow X^{1})$ and such that $\mathcal{Y}\otimes_{R}-$ leaves the sequence exact. Let $W$ be in $\mathcal{X}\cap \mathcal{X}^{\perp}$. Then we have a pure exact sequence sequence $0\rightarrow W\rightarrow W^{++}\rightarrow W^{++}/W\rightarrow 0$
in $R$-Mod by \cite[Proposition 5.3.9]{EJ00}. Since $W\in \mathcal{X}$ and $(\mathcal{X},\mathcal{Y})$ is a symmetric duality pair, we get $W^{++}\in \mathcal{X}$. Thus $W^{++}/W\in \mathcal{X}$, and so the above pure exact sequence is split. From the fact that $\Hom_{R}(\mathbb{X}, W^{++})\cong(W^{+}\otimes_{R}\mathbb{X})^{+}$, and $W^{+}\otimes_{R}\mathbb{X}$ is exact by hypothesis, we obtain that the sequence $\Hom_{R}(\mathbb{X}, W^{++})$ is exact. Therefore, $\Hom_{R}(\mathbb{X}, W)$ is exact.

(2)$\Rightarrow$ (1) By (2), there exists an exact sequence of left $R$-modules in $\mathcal{X}$,
$$\mathbb{X}:\cdots\rightarrow X_{1}\rightarrow X_{0}\rightarrow X^{0}\rightarrow X^{1}\rightarrow\cdots$$
such that $M\cong \Ker(X^{0}\rightarrow X^{1})$ and such that $\Hom_{R}(-, \mathcal{X}\cap \mathcal{X}^{\perp})$ leaves the sequence exact. Let $Y\in \mathcal{Y}$. Note that $(Y\otimes_{R}\mathbb{X})^{+}\cong\Hom_{R}(\mathbb{X}, Y^{+})$, and $\Ext_{R}^{1}(X, Y^{+})\cong\Tor^{R}_{1}(X,Y)^{+}$. Since $\mathcal{X}\top\mathcal{Y}$, and $(\mathcal{X},\mathcal{Y})$ is a symmetric duality pair, we get that $Y^{+}\in \mathcal{X}\cap \mathcal{X}^{\perp}$. Thus $\Hom_{R}(\mathbb{X}, Y^{+})$ is exact by hypothesis, and so $Y\otimes_{R}\mathbb{X}$ is exact. Hence $M$ is in $\mathcal{GF}_{\mathcal{(X,Y)}}(R)$.

(2)$\Leftrightarrow$ (3) and (4)$\Leftrightarrow$ (5) are straightforward.

(4)$\Rightarrow$ (3) is clear.

(3)$\Rightarrow$ (4) By (3), there exists an exact sequence $0\rightarrow M\rightarrow X^{0}\rightarrow L^{1}\rightarrow 0$ with $X^{0}\in \mathcal{X}$ and $L^{1}\in \mathcal{GF}_{\mathcal{(X,Y)}}(R)$. By Lemma \ref{lem1.1}, we get that ($\mathcal{X}, \mathcal{X}^{\bot}$) is a perfect  cotorsion pair, and so it is complete. Then  there exists an exact sequence $0\rightarrow X^{0}\rightarrow W^{0}\rightarrow N^{1}\rightarrow 0$ with $W^{0}\in \mathcal{X}\cap \mathcal{X}^{\perp}$ and $N^{1}\in \mathcal{X}$. Consider the following pushout diagram
$$\xymatrix{&& 0   \ar[d]_{ }& 0 \ar[d]_{ }   &  \\
0 \ar[r]^{} & M \ar@{=}[d]  \ar[r]^{ } &X^{0}  \ar[d]_{ } \ar[r]^{ }
& L^1\ar[d]_{ } \ar[r]^{ } &0 \\
0 \ar[r]^{} & M\ar[r]^{ } &W^{0}\ar[d]_{} \ar[r]^{} & U^1 \ar[d]_{ } \ar[r]^{ } &0 \\
&  &N^{1}\ar[d]_{}\ar@{=}[r]^{}& N^{1} \ar[d]_{ }  & \\
&    &0             & 0        &   }$$
In the right column, since $L^1\in \mathcal{GF}_{\mathcal{(X,Y)}}(R)$ and $N^1\in \mathcal{X}$, we get that $U^1\in \mathcal{GF}_{\mathcal{(X,Y)}}(R)$ by Lemma \ref{lem2.4}. So $\Ext_{R}^{i}(U^1, W)=0$ for all $W\in \mathcal{X}\cap \mathcal{X}^{\perp}$ and the exact sequence $0\rightarrow M\rightarrow W^{0}\rightarrow U^{1}\rightarrow 0$ is $\Hom_{R}(-, \mathcal{X}\cap \mathcal{X}^{\perp})$-exact. By repeating the proceeding process, we have a $\Hom_{R}(-, \mathcal{X}\cap \mathcal{X}^{\perp})$-exact exact sequence $$0\rightarrow M\rightarrow W^{0}\rightarrow W^{1}\rightarrow\cdots$$
with each $W^{i}\in \mathcal{X}\cap \mathcal{X}^{\perp}$.
\end{proof}

\begin{corollary}\label{cor2.5} Assume that $(\mathcal{X},\mathcal{Y})$ is a perfect and symmetric duality pair, $\mathcal{X}$ is closed under kernels of epimorphisms, and $0\rightarrow L\rightarrow M\rightarrow N\rightarrow 0$ is a short exact sequence in $R$-Mod.

$\mathrm{(1)}$ If $N\in \mathcal{GF}_{\mathcal{(X,Y)}}(R)$, then $L\in \mathcal{GF}_{\mathcal{(X,Y)}}(R)$ if and only if $M\in \mathcal{GF}_{\mathcal{(X,Y)}}(R)$.

$\mathrm{(2)}$ If $L, M\in \mathcal{GF}_{\mathcal{(X,Y)}}(R)$, then $N\in \mathcal{GF}_{\mathcal{(X,Y)}}(R)$ if and only if $\Ext_{R}^{1}(N, \mathcal{X}\cap \mathcal{X}^{\perp})=0$.
\end{corollary}
\begin{proof}By Proposition \ref{prop2.3}, it is obtained by standard argument.
\end{proof}
\begin{remark}(1) Under the condition of Proposition \ref{prop2.3}, the class $\mathcal{GF}_{\mathcal{(X,Y)}}(R)$ is closed under extensions by Corollary \ref{cor2.5}, in this case, Theorem \ref{the2.1} and Theorem \ref{the2.2} also hold.

(2) Under the condition of Proposition \ref{prop2.3}, by Propositions \ref{prop2.1} and \ref{prop2.3}, and Corollary \ref{cor2.5} we have that the class $\mathcal{GF}_{\mathcal{(X,Y)}}(R)$ is projectively resolving, and $\mathcal{X}\cap \mathcal{X}^{\perp}$ is a cogenerator for $\mathcal{GF}_{\mathcal{(X,Y)}}(R)$, and $\Ext_{R}^{i}(G, W)=0 $ for all $G\in \mathcal{GF}_{\mathcal{(X,Y)}}(R)$ and $W\in \mathcal{X}\cap \mathcal{X}^{\perp}$. In addition, it is clear that $\mathcal{X}\cap \mathcal{X}^{\perp}$ is closed under direct summands. According to \cite[Theorem 3.1]{zhu13}, we can get some Ext-functorial descriptions of Gorenstein $(\mathcal{X},\mathcal{Y})$-flat dimension.
\end{remark}
\section{The Gorenstein flat model structures with respect to duality pairs}
According to Hovey's correspondence \cite[Theorem 2.2]{Hov02}, we know that an abelian model structure on $R$-Mod, in fact on any abelian category, is equivalent to a triple ($\mathcal{Q}, \mathcal{W}, \mathcal{R}$) of classes of objects for which $\mathcal{W}$ is thick and ($\mathcal{Q}\cap\mathcal{W}, \mathcal{R}$) and ($\mathcal{Q}, \mathcal{W}\cap\mathcal{R}$) are each complete cotorsion pairs. We say that $\mathcal{W}$ is thick if the class $\mathcal{W}$ is closed under extensions, direct summands, kernels of epimorphisms, and cokernels of monomorphisms. In this case, $\mathcal{Q}$ is precisely the class of cofibrant objects, $\mathcal{R}$ is precisely fibrant objects, and $\mathcal{W}$ is the class of trivial objects of the model structure. Hence, we denote an abelian model structure as a triple $\mathcal{M}=$($\mathcal{Q}, \mathcal{W}, \mathcal{R}$), and call it a Hovey triple, and denote the two associated cotorsion pairs above by ($\mathcal{\tilde{Q}}, \mathcal{R}$) and ($\mathcal{Q}, \mathcal{\tilde{R}}$), where $\mathcal{\tilde{Q}}=\mathcal{Q}\cap \mathcal{W}$ is the class of trivially cofibrant objects and $\mathcal{\tilde{R}}=\mathcal{R}\cap \mathcal{W}$ is the class of trivially fibrant objects. We say that $\mathcal{M}$ is hereditary if both of these associated cotorsion pairs are hereditary. We refer to \cite{Gil16,Hov02,Hov99} for a more detailed discussion on model structures.

In \cite{Gil15}, the main theorem gave a new method for constructing hereditary abelian model structures, as follows.

\begin{theorem}(\cite[Theorem 1.1]{Gil15})\label{the3.1} Let ($\mathcal{\tilde{Q}}, \mathcal{R}$) and ($\mathcal{Q}, \mathcal{\tilde{R}}$) be two complete hereditary cotorsion pairs in an abelian category $\mathcal{C}$ satisfying the two conditions below.

$\mathrm{(1)}$ $\mathcal{\tilde{R}}\subseteq\mathcal{R}$ and $\mathcal{\tilde{Q}}\subseteq\mathcal{Q}$.

$\mathrm{(2)}$ $\mathcal{\tilde{Q}}\cap\mathcal{R}=\mathcal{Q}\cap\mathcal{\tilde{R}}$.\\
Then ($\mathcal{Q}, \mathcal{W}, \mathcal{R}$) is a Hovey triple, where the thick class $\mathcal{W}$ can be described in two following ways:

\[\begin{aligned}\mathcal{W}&=\{X\in \mathcal{C}~|~\text{there exists a short exact sequence}~0\rightarrow X\rightarrow R\rightarrow Q\rightarrow 0 ~\text{with}~R\in \tilde{R}, Q\in \tilde{Q}\}\\
&=\{X\in \mathcal{C}~|~\text{there exists a short exact sequence}~0\rightarrow Q'\rightarrow R'\rightarrow X \rightarrow 0 ~\text{with}~R'\in \tilde{R}, Q'\in \tilde{Q}\}.\end{aligned}.\]
Moreover, $\mathcal{W}$ is unique in the sense that if $\mathcal{V}$ is another thick class for which ($\mathcal{Q}, \mathcal{V}, \mathcal{R}$) is a Hovey triple, then necessarily $\mathcal{V}=\mathcal{W}$.
\end{theorem}
In \cite{Gil17}, over a right coherent ring $R$, Gillespie constructed a hereditary abelian model structure on $R$-Mod whose cofibrant objects are precisely the Gorenstein flat modules using the above theorem. In \cite{EIP17}, Estrada, Iacob and P\'{e}res studied relative Gorenstein flat model structure on the categories of left $R$-modules and complexes. In the section, using the same method, we construct an abelian model structure on $R$-Mod whose cofibrant objects are precisely the Gorenstein $(\mathcal{X},\mathcal{Y})$-flat modules, where $(\mathcal{X},\mathcal{Y})$ is a duality pair.

 Recall that a class $\mathcal{A}$ of left $R$-modules is a $\kappa$-Kaplansky class if there exists a cardinal number $\kappa$ such that for every $M\in \mathcal{A}$  and for any subset $S\subseteq M$ with Card$(S)\leqslant\kappa$, there exists a submodule $N$ of $M$ that contains $S$ with the property that Card$(N)\leqslant\kappa$ and both $N$ and $M/N$ are in $\mathcal{A}$. We say that $\mathcal{A}$ is a Kaplansky class if it is a $\kappa$-Kaplansky class for some regular cardinal $\kappa$. Similarly, we have the notion of a Kaplansky class of chain complexes.

 \begin{lemma}\label{ER02}(\cite[Theorem 2.9]{ER02}) Let $\mathcal{K}$ be a Kaplansky class. If $\mathcal{K}$ contains the projective modules and it is closed under extensions and direct limits, then ($\mathcal{K}$, $\mathcal{K}^{\bot}$) is a perfect cotorsion pair in $R$-Mod.
\end{lemma}

\begin{lemma}\label{lem3.1} If $\mathcal{X}$ is closed under pure submodules and pure quotients, then the class $\mathcal{GF}_{\mathcal{(X,Y)}}(R)$ is a Kaplansky class.
\end{lemma}

\begin{proof}Let $_{\mathcal{Y}}\mathcal{\widetilde{EX}}$ be the class of exact complexes of left $R$-modules in $\mathcal{X}$ such that $\mathcal{Y}\otimes_{R}-$ leaves it exact. Similar to the proof of \cite[Theorem 3.7]{EG15}, we get that the class $_{\mathcal{Y}}\mathcal{\widetilde{EX}}$ is closed under pure subcomplexes and pure quotients, and so it is a Kaplansky class by \cite[Proposition 3.4]{EG15}. Since $\mathcal{GF}_{\mathcal{(X,Y)}}(R)$ is the class of $0$-cycles of complexes in $_{\mathcal{Y}}\mathcal{\widetilde{EX}}$, we obtain that it is a Kaplansky class.
\end{proof}

\begin{lemma}\label{lem3.2} Assume that $\mathcal{X}$ is closed under direct sums, pure submodules and pure quotients, and the class $\mathcal{GF}_{\mathcal{(X,Y)}}(R)$ is closed under extensions. Then the class $\mathcal{GF}_{\mathcal{(X,Y)}}(R)$ is closed under direct limits.
\end{lemma}
\begin{proof}Since $\mathcal{X}$ is closed under direct sums, pure submodules and pure quotients, we have that $\mathcal{X}$ is closed under direct limits. It is obtained by argument similar to the proof of \cite[Lemma 3.1]{YL12-2}.
\end{proof}

\begin{proposition}Assume that $\mathcal{X}$ is closed under direct sums, pure submodules and pure quotients, and the class $\mathcal{GF}_{\mathcal{(X,Y)}}(R)$ is closed under extensions and direct products. Then the class $\mathcal{GF}_{\mathcal{(X,Y)}}(R)$ is preenveloping.
\end{proposition}
\begin{proof}By Lemma \ref{lem3.1} and \ref{lem3.2}, we get that
$\mathcal{GF}_{\mathcal{(X,Y)}}(R)$ is a Kaplansky class, and is closed under direct limits. Then the result follows from \cite[Theorem 2.5]{ER02}.
\end{proof}

\begin{proposition}\label{prop3.1} Assume that $\mathcal{X}$ is closed under direct sums, pure submodules and pure quotients, and the class $\mathcal{GF}_{\mathcal{(X,Y)}}(R)$ is closed under extensions. Then $(\mathcal{GF}_{\mathcal{(X,Y)}}(R), \mathcal{GF}_{\mathcal{(X,Y)}}(R)^{\bot})$ is a perfect and hereditary cotorsion pair.
\end{proposition}

\begin{proof}According to Lammas \ref{ER02}, \ref{lem3.1}, and \ref{lem3.2}, we get that $(\mathcal{GF}_{\mathcal{(X,Y)}}(R), \mathcal{GF}_{\mathcal{(X,Y)}}(R)^{\bot})$ is a perfect cotorsion pair. Since the class $\mathcal{GF}_{\mathcal{(X,Y)}}(R)$ is projectively resolving by Proposition \ref{prop2.1}, we have that $(\mathcal{GF}_{\mathcal{(X,Y)}}(R), \mathcal{GF}_{\mathcal{(X,Y)}}(R)^{\bot})$ is hereditary.
\end{proof}

\begin{corollary} Assume that $\mathcal{X}$ is closed under direct sums, pure submodules and pure quotients, and the class $\mathcal{GF}_{\mathcal{(X,Y)}}(R)$ is closed under extensions. Then the class $\mathcal{GF}_{\mathcal{(X,Y)}}(R)$ is covering.
\end{corollary}

When $(\mathcal{X},\mathcal{Y})$ is a perfect and symmetric duality pair, and $\mathcal{X}$ is closed under kernels of epimorphisms. Then, by Lemmas \ref{lem1.1} and \ref{lem1.2}, we get that ($\mathcal{X}$, $\mathcal{X}^{\bot}$) is a complete (perfect) and hereditary cotorsion pair, and $\mathcal{X}$ is closed under direct sums, pure submodules and pure quotients. By Corollary \ref{cor2.5}, we have that the class $\mathcal{GF}_{\mathcal{(X,Y)}}(R)$ is closed under extensions. It is clear that $\mathcal{X}\subseteq\mathcal{GF}_{\mathcal{(X,Y)}}(R)$ and $\mathcal{GF}_{\mathcal{(X,Y)}}(R)^{\bot}\subseteq\mathcal{X}^{\bot}$. The desired Hovey triple shall be a consequence of the following result.

\begin{proposition}\label{prop3.3} Assume that $(\mathcal{X},\mathcal{Y})$ is a perfect and symmetric duality pair, and $\mathcal{X}$ is closed under kernels of epimorphisms. Then the cotorsion pairs ($\mathcal{X}$, $\mathcal{X}^{\bot}$) and $(\mathcal{GF}_{\mathcal{(X,Y)}}(R), \mathcal{GF}_{\mathcal{(X,Y)}}(R)^{\bot})$ have the same core. That is, $\mathcal{GF}_{\mathcal{(X,Y)}}(R)\cap\mathcal{GF}_{\mathcal{(X,Y)}}(R)^{\bot}
=\mathcal{X}\cap\mathcal{X}^{\bot}$.
\end{proposition}
\begin{proof}($\subseteq$) Suppose $W\in \mathcal{GF}_{\mathcal{(X,Y)}}(R)\cap\mathcal{GF}_{\mathcal{(X,Y)}}(R)^{\bot}$. Since $\mathcal{GF}_{\mathcal{(X,Y)}}(R)^{\bot}\subseteq\mathcal{X}^{\bot}$, we only have to show that $W$ is in $\mathcal{X}$. Since $W$ is in $\mathcal{GF}_{\mathcal{(X,Y)}}(R)$, we have a short exact sequence $0\rightarrow W\rightarrow X\rightarrow W'\rightarrow 0$, where $X$ is in $\mathcal{X}$ and $W'$ is in $\mathcal{GF}_{\mathcal{(X,Y)}}(R)$. According to $W\in \mathcal{GF}_{\mathcal{(X,Y)}}(R)^{\bot}$, we get that $\Ext_{R}^{1}(W', W)=0$, and so the above sequence splits. Therefore, $W$ is a direct summand of $X$, which implies $W\in \mathcal{X}$.

($\supseteq$) Let $U\in \mathcal{X}\cap\mathcal{X}^{\bot}$. Then, it is clear that $U$ is in $\mathcal{GF}_{\mathcal{(X,Y)}}(R)$. It needs to prove that $U\in \mathcal{GF}_{\mathcal{(X,Y)}}(R)^{\bot}$. First, we show that the sequence $\Hom_{R}(\mathbb{X}, U)$ is exact for any $\mathbb{X}\in{_{\mathcal{Y}}\widetilde{\mathcal{EX}}}$. Note that $\Hom_{R}(\mathbb{X}, U^{++})\cong(U^{+}\otimes_{R}\mathbb{X})^{+}$. Since $U\in \mathcal{X}$, we get that $U^{+}\in \mathcal{Y}$, and so $U^{+}\otimes_{R}\mathbb{X}$ is exact. Hence, $\Hom_{R}(\mathbb{X}, U^{++})$ is exact. From \cite[Proposition 5.3.9]{EJ00}, there exists a pure exact sequence $0\rightarrow U\rightarrow U^{++}\rightarrow U^{++}/U\rightarrow 0$.
Since $\mathcal{X}$ is closed under pure quotients, we get that $U^{++}/U\in \mathcal{X}$. By the assumption that $U$ is also in $\mathcal{X}^{\bot}$, which means that the above sequence splits. Thus $\Hom_{R}(\mathbb{X}, U)$ is exact. Now we can prove by the standard argument that $\Ext_{R}^{1}(G, U)=0$ for all $G \in \mathcal{GF}_{\mathcal{(X,Y)}}(R)$.
\end{proof}

Combine Proposition \ref{prop3.3} and Theorem \ref{the3.1}, we have a Hovey triple ($\mathcal{GF}_{\mathcal{(X,Y)}}(R), \mathcal{W}, \mathcal{X}^{\bot}$) and the following result.

\begin{theorem}\label{the3.9} Assume that $(\mathcal{X},\mathcal{Y})$ is a perfect and symmetric duality pair, and $\mathcal{X}$ is closed under kernels of epimorphisms. Then the category $R$-Mod of left $R$-modules has a hereditary abelian model structure, the Gorenstein flat model structure with respect to a duality pair $(\mathcal{X}, \mathcal{Y})$, as follows:

 The cofibrant objects coincide with the clas $\mathcal{GF}_{\mathcal{(X,Y)}}(R)$.

 The fibrant objects coincide with the class $\mathcal{X}^{\bot}$.

 The trivially cofibrant objects coincide with the class $\mathcal{X}$.

 The trivially fibrant objects coincide with the class $\mathcal{GF}_{\mathcal{(X,Y)}}(R)^{\bot}$.\\
 An $R$-module $M$ fits into a short exact sequence $0\rightarrow G\rightarrow X\rightarrow M\rightarrow 0$ with $X\in \mathcal{X}$ and $G\in \mathcal{GF}_{\mathcal{(X,Y)}}(R)^{\bot}$ if and only if it fits into a short exact sequence $0\rightarrow M\rightarrow X'\rightarrow G'\rightarrow 0$ with $X'\in \mathcal{X}$ and $G'\in \mathcal{GF}_{\mathcal{(X,Y)}}(R)^{\bot}$. All moudles $M$ with this property are precisely the trivial objects of the above abelian model structure.

\end{theorem}

Applying Theorem \ref{the3.9} to Example \ref{ex2.1}(1), we have  \cite[Theorem 3.3]{Gil17}.

%\vspace{0.3cm} \hspace{-0.8cm}{\bf{Acknowledgement}}

%The authors would like to express their sincere thanks to the referee for his/her helpful  suggestions and comments, which have greatly improved the paper.

\def\refname


{\hfil
References}
\begin{thebibliography}{ABCDF}

\bibitem{Ben09} D. Bennis, Rings over which the class of Gorenstein flat modules is closed under extensions, Comm. Algebra, 2009, 37:855--868.
\bibitem{BEI17}D. Bravo, S. Estrada, A.Iacob, FP$_{n}$-injective and FP$_{n}$-flat covers and envelopes, and Gorenstein AC-flat covers, Arxiv:1709.10160v1, 2017.
\bibitem{BGH14}D. Bravo, J. Gillespie, M. Hovey, The stable model category of a general ring, arxiv:1405.5768, 2014.
\bibitem{BK12} S. Bouchiba., M. Khaloui, Stability of Gorenstein flat modules. Glasg. Math. J., 2012, 54:169--175.
\bibitem{BP17}D. Bravo, M.A. P\'{e}rez, Finiteness conditions and cotorsion pairs, J. Pure Appl. Algebra, 2017, 221(6):1249--1267.
\bibitem{EG15} S. Estrada, J.Gillespie, The projective stable category of a coherent scheme, arxiv:1511.02907, 2015.

\bibitem{EIP17} S. Estrada, A. Iacob, M. A. P\'{e}res, Model structures and relative Gorenstein flat modules, arxiv:1709.00658, 2017.
\bibitem{EJ95} E.E. Enochs, O.M.G. Jenda. Gorenstein injective and projective modules. Math. Z., 1995, 220:611--633.
\bibitem{EJT93} E.E. Enochs, O.M.G. Jenda, B. Torrecillas. Gorenstein flat modules. Nanjing Daxue Xuebao Shuxue Bannian Kan, 1993, 10:1--9.
\bibitem{EJ00} E.E. Enochs, O.M.G. Jenda. Relative Homological Algebra, Walter de Gruyter, Berlin, 2000.
\bibitem{ER02} E.E. Enochs, J.A. L\'{o}pez Ramos, Kaplansky classes, Rend. Semin. Mat. Univ. Padova, 2002, 107:67--79.
\bibitem{Gil15}  J. Gillespie, How to construct a Hovey triple from two cotorsion pairs, Fundam. Math., 2015, 230(3):281--289.
\bibitem{Gil16} J. Gillespie, Hereditary abelian model categories, Bull. London Math. Soc., 2016, 48(6):895--922.
\bibitem{Gil17}  J. Gillespie, Duality pairs and stable module categories, arxiv:1710.09906v1.
\bibitem{Hol04} H. Holm, Gorenstein homological dimensions, J. Pure Appl. Algebra, 2004, 189:167--193.
\bibitem{HJ09}H. Holm, P. J{\o}rgensen, Cotorsion pairs induced by duality pairs, J. Commut. Algebra, 2009, 1(4):621--633.

\bibitem{Hov02} M. Hovey, Cotorsion pairs, model category structures, and representation theory, Math. Z., 2002, 241(3):553--592.
\bibitem{Hov99} M. Hovey, Model Categories, American Mathematical Society, Providence, RI,1999.
\bibitem{Hu13} Z.Y. Huang. Proper resolutions and Gorenstein categories, J. Algebra, 2013, 393:142--169.
\bibitem{M07} L.X. Mao, Min-flat modules and min-coherent rings, 2007, 35:635--650.
\bibitem{MD08} L.X. Mao, N.Q. Ding, Gorenstein FP-injective and Gorenstein flat modules, J. Alg. Appl., 2008, 7(4):491--506.
\bibitem{SUU14} C. Selvaraj, R. Udhayakumar, A. Umamaheswaran, Gorenstein $n$-flat modules and their covers, Asian-European J. Math.,2014,7:1450051(13pages).
\bibitem{SWW08} S. Sather-Wagstaff, T. Sharif, D. White. Stability of Gorenstein categories. J. London Math. Soc., 2008, 77:481--502.
\bibitem{WGM17} Z.P. Wang, S.T. Guo, H.Y. Ma, Stability of Gorenstein modules with respect to a semidualizing module, J. of Math., 2017, 37:1143--1153.
\bibitem{Whi10}D. White. Gorenstein projective dimension with respect to a semidualizing module. J. Commut. Algebra, 2010, 2:111--137.
\bibitem{XD13} A.M. Xu, N.Q. Ding. On stability of Gorenstein categories. Comm. Algebra, 2013, 41:3793--3804.
\bibitem{YL12-1} G. Yang, Z.K. Liu.  Stability of Gorenstein flat categories. Glasg. Math. J., 2012, 54:177--191.
\bibitem{YL12-2}  G. Yang, Z.K. Liu.  Gorenstein flat covers over GF-closed rings. Comm. Algebra, 2012, 40:1632--1640.
\bibitem{zhu13}X.S. Zhu, Resolving resolution dimensions, Algebr. Represent. Theory, 2013, 16:1165--1191.

\end{thebibliography}
\end{document}